\documentclass{amsart}
\usepackage{amssymb}
\usepackage{hyperref}
\newtheorem{theorem}{Theorem}%[section]
\DeclareMathOperator{\supp}{supp}

\title{Fourier dimension of the cone}
\author{Terence L.~J.~Harris}
\address{Department of Mathematics, Cornell University, Ithaca, NY 14853, USA}
\email{tlh236@cornell.edu}
\begin{document} 
\begin{abstract} It is shown that the cone in $\mathbb{R}^{d+1}$ has Fourier dimension $d-1$. This verifies a conjecture of Fraser and Kroon. 
\end{abstract}
\maketitle
The Fourier dimension of a set $A \subseteq \mathbb{R}^{d+1}$ is defined by
\[ \dim_F A = \sup\left\{ s \in [0,d+1]: \exists \mu \in \mathcal{P}(A) \text{ s.t.~} \lvert\widehat{\mu}(\xi) \rvert \lesssim \lvert \xi \rvert^{-s/2} \; \forall \xi \in \mathbb{R}^{d+1} \right\}, \]
where $\mathcal{P}(A)$ is the set of Borel probability measures on $\mathbb{R}^{d+1}$ satisfying $\mu(A)=1$. 
In \cite[Problem~6.1]{fraserkroon} it was conjectured that for $d \geq 1$, the cone
\[ \Gamma^d = \left\{ (\xi_1, \dotsc, \xi_d, \xi_{d+1}) \in \mathbb{R}^{d+1}:  \lvert (\xi_1, \dotsc, \xi_d ) \rvert = \vert \xi_{d+1} \rvert \right\}, \]
has Fourier dimension equal to $d-1$. This is verified by the following.
\begin{theorem} For any $d \geq 1$, the Fourier dimension of $\Gamma^d$ is $d-1$. \end{theorem} 
\begin{proof} The case $d =1$ is trivial, so assume that $d \geq 2$. The lower bound $\dim_F \Gamma \geq d-1$ follows (for example) by using \eqref{sphereasymp} below and considering the measure defined by 
\[ f \mapsto \int_{\mathbb{R}} \psi(r) \int_{S^{d-1}} f(r x, r) \, d\sigma(x) \, dr, \]
for any non-negative Borel function $f$, where $\sigma$ is the rotation invariant Borel probability measure on $S^{d-1}$, and $\psi$ is a bump function on $[1,2]$ with $\int \psi = 1$. 

Suppose for a contradiction that $\dim_F \Gamma > d-1$. Then %by Lemma~1 from Ekstrom-Persson-Schmeling 
there exists $\alpha >d-1$ %, $\epsilon >0$
 and a Borel probability measure $\mu$ on $\Gamma$, such that 
\begin{equation} \label{mudecay}  \lvert\widehat{\mu}(\xi) \rvert  \lesssim \lvert\xi\rvert^{-\alpha/2} \qquad\forall \xi \in \mathbb{R}^{d+1}. \end{equation}
By symmetry, and by replacing $\mu$ with $f \mu$ for an appropriate bump function $f$ (see \cite[Lemma~1]{ekstrom}), it may be assumed that for some $\epsilon>0$, 
\begin{equation} \label{sptcdn} \supp \mu \subseteq \{ (\xi, \lvert\xi\rvert ) \in \mathbb{R}^d \times \mathbb{R}: \epsilon \leq \lvert\xi\rvert  \leq 1/\epsilon \}. \end{equation}
 Let $\nu$ be the Borel probability measure on $\Gamma$ defined by
\begin{equation} \label{nudefn} \int f \, d\nu = \int_{\mathbb{R}^d \times  \mathbb{R}}  \int_{S^{d-1}} f(\lvert x\rvert w, z) \, d\sigma(w) \, d\mu(x, z), \end{equation}
for any non-negative Borel function $f$. Then 
\begin{align*} \widehat{\nu}(\xi) &= \int_{\mathbb{R}^d \times \mathbb{R}} \int_{S^{d-1}} e^{-2\pi i \langle \xi, (\lvert x\rvert w  , z) \rangle } \, d\sigma(w) \, d\mu(x,z)   \\
&=  \int_{\mathbb{R}^d \times \mathbb{R}} \int_{O(d)} e^{-2\pi i \langle \xi, (Ux , z) \rangle } \, d\lambda(U) \, d\mu(x,z) \\
&= \int_{O(d)} \widehat{ \mu} (U^*(\xi_1,\dotsc ,\xi_d), \xi_{d+1} ) \, d\lambda(U),  \end{align*} 
where $\lambda$ is the Haar probability measure on $O(d)$. Hence $\nu$ satisfies
\begin{equation} \label{nudecay} \lvert \widehat{\nu} (\xi)\rvert   + \left\lvert \widehat{\widetilde{\nu}} (\xi) \right\rvert\lesssim \lvert \xi\rvert^{-\alpha/2} \qquad\forall \xi \in \mathbb{R}^{d+1}, \end{equation}
where $\widetilde{\nu}$ is the pushforward of $\nu$ under $(x_1, \dotsc, x_d, x_{d+1}) \mapsto (x_1, \dotsc, x_d, -x_{d+1})$.  Let $\pi: \mathbb{R}^{d+1} \to \mathbb{R}$ be the map $(x_1,\dotsc,x_d,x_{d+1}) \mapsto x_{d+1}$. Since $\supp \mu \subseteq \Gamma$, and by \eqref{sptcdn}, the formula \eqref{nudefn} can also be written as 
\[ \int f \, d\nu = \int_{\mathbb{R}} \int_{S^{d-1}} f(zw, z) \, d\sigma(w) \, d\pi_{\#}\mu(z). \] 
Hence another expression for $\widehat{\nu}$ is
\[ \widehat{\nu}(\xi) = \int e^{-2\pi i z \xi_{d+1} } \widehat{\sigma}(z (\xi_1, \dotsc, \xi_d)) \, d\pi_{\#}\mu(z). \]
Let $\omega_{d-1}$ be the surface area of $S^{d-1}$. The asymptotic formula (see \cite[Appendix~B]{grafakos})
\begin{equation} \label{sphereasymp} \widehat{\sigma}(\xi) = \frac{2}{\omega_{d-1}}\lvert \xi\rvert^{-(d-1)/2} \cos\left(2 \pi \lvert \xi\rvert  - \frac{\pi(d-1)}{4}  \right) + O\left(\lvert \xi\rvert^{-(d+1)/2}\right), \end{equation}
gives, by taking $\xi_{d+1} = \lvert (\xi_1, \dotsc, \xi_d)\rvert$, 
\begin{multline*}  e^{ \frac{i\pi(d-1)}{4} } \widehat{\nu}(\xi, \lvert \xi\rvert ) +  e^{ \frac{-i\pi(d-1)}{4} }\widehat{ \widetilde{\nu} }(\xi, \lvert \xi\rvert )  \\
= \frac{4}{\omega_{d-1}}  \lvert \xi\rvert^{-(d-1)/2} \int z^{-(d-1)/2} \cos^2\left( 2 \pi \lvert \xi\rvert  z  - \frac{\pi(d-1)}{4} \right) \, d\pi_{\#}\mu(z) \\ +O\left(\lvert \xi\rvert^{-(d+1)/2}\right). \end{multline*}
Comparing \eqref{nudecay} to the above will give a contradiction, by the following identity:
\begin{equation} \label{halfidentity} \lim_{r \to \infty} \int \cos^2\left(r z+t \right) \, d\pi_{\#}\mu(z) = 1/2 \qquad\forall t \in \mathbb{R}.  \end{equation}
It remains to prove \eqref{halfidentity}. Since $d \geq 2$ and $\alpha/2 > (d-1)/2 \geq 1/2$, condition \eqref{mudecay} gives $\pi_{\#} \mu \in L^2(\mathbb{R})$ (see e.g.~\cite[Theorem~3.3]{mattila}). Hence $\pi_{\#} \mu \in L^1(\mathbb{R})$ with $\lVert \pi_{\#} \mu \rVert_1=1$, and \eqref{halfidentity} then follows by approximating $\pi_{\#} \mu$ in $L^1$ with a finite linear combination of characteristic functions of disjoint intervals.  \end{proof}

\end{document}